\documentclass[10pt,a4paper,reqno]{amsart}

\usepackage[english]{babel}
\usepackage{latexsym}
\usepackage{amsmath}
\usepackage{amssymb}
\usepackage{amsthm}

\newtheorem{Th}{Theorem}[section]

\newtheorem{Lemma}[Th]{Lemma}
\newtheorem{Def}[Th]{Definition}

\newcommand{\lra}{\longrightarrow}
\newcommand{\lmt}{\longmapsto}

\newcommand{\x}{\times}
\newcommand{\ox}{\otimes}

\newcommand{\bigox}{\bigotimes}

\newcommand{\od}{\odot}
\newcommand{\bod}{\bigodot}
\newcommand{\w}{\wedge}
\renewcommand{\v}{\vee}
\newcommand{\R}{\mathbb{R}}
\newcommand{\N}{\mathbb{N}}

\newcommand{\Po}{\mathcal{P}}

\newcommand{\Li}{\mathcal{L}}

\newcommand{\e}{\varepsilon }

  \newcommand{\bx}{\mathrm{\bf x}}

\begin{document}
\title[A representation theorem for orthogonally additive polynomials]{A representation theorem for orthogonally additive polynomials on Riesz spaces}

\author{A. Ibort, P. Linares}
\address{Departamento de Matem{\'a}ticas, Universidad Carlos III de
  Madrid, Avda. de la Universidad 30, 28911 Legan{\'e}s, Spain}

\email{albertoi@math.uc3m.es}
\thanks{The first author was supported in part by Project MTM 2007-62478.
The second author was partially supported by the "Programa de formaci{\'o}n del profesorado
universitario del MEC". The second and third author were supported
in part by Project MTM 2006-03531.}
\thanks{The authors wish to thank the suggestions and observations of the referee that have helped greatly to shape the final form of this article}

\author{J.G. Llavona}
\address{Departamento de An{\'a}lisis Matem{\'a}tico, Facultad de Matem{\'a}ticas, Universidad Complutense de  Madrid, 28040 Madrid, Spain}
\email{jl\_llavona@mat.ucm.es,  plinares@mat.ucm.es}
\thanks{}


\begin{abstract}
The aim of this article is to prove a representation theorem for orthogonally additive polynomials in the spirit of the recent theorem on representation of orthogonally additive polynomials on Banach lattices but for the setting of Riesz spaces. To this purpose the notion of $p$--orthosymmetric multilinear form is introduced and it is shown to be equivalent to the or\-tho\-go\-na\-lly additive property of the corresponding polynomial.    
Then the space of positive orthogonally additive polynomials on an Archimedean Riesz space taking values on an uniformly complete Archimedean Riesz space is shown to be isomorphic to the space of positive linear forms on the $n$-power in the sense of Boulabiar and Buskes of the original Riesz space.
\end{abstract}

\maketitle

\section{Introduction}

Given $X, Y$ two Banach lattices, a $n$-homogeneous polynomial $P\in\Po(^nX,Y)$ is said to be an orthogonally additive polynomial if $P(x+y)=P(x)+P(y)$ whenever $x$ and $y$ are disjoint elements of $X$. The first mathematician interested in that kind of polynomials was Sundaresan who, in 1991, obtained a representation theorem for polynomials on $\ell_{p}$ and on $L^{p}$.  P{\'e}rez-Garc{\'\i}a and Villanueva \cite{Perez Garcia-Villanueva 2005} proved the case $C(K)$ and the result was generalized to any Banach lattice by Benyamini, Lassalle and Llavona
\cite{Benyamini-Lassalle-Llavona 2006}. Carando, Lassalle and Zalduendo \cite{Carando-Lassalle-Zalduendo 2006} found an independent proof for the case $C(K)$ and Ibort, Linares and Llavona \cite{Ibort-Linares-Llavona 2009} gave another one for the case $\ell_{p}$.

In parallel to the study of the orthogonally additive polynomials in Banach lattices, there has been several authors interested in the study of an analogous concept for
multilinear forms in Riesz spaces.

Buskes and van Rooij introduced in 2000 the concept of orthosymmetric bilinear mapping in a Riesz space (see \cite{Buskes-Van Rooij1 2000}) They showed that the orthosymmetric
positive mappings are symmetric which lead to a new proof of the commutativity of $f$-algebras.

In 2004 the same authors studied the representation of orthosymmetric mappings in Riesz spaces introducing the notion of the square of a Riesz space and giving several characterizations of it in their article \cite{Buskes-Van Rooij 2004}. This study was generalized to the $n$-linear case by Boulabiar and Buskes \cite{Boulabiar-Buskes 2006} in 2006.

The square works in a similar way as the concavification for Banach lattices, so it naturally  raises the problem of the representation of orthogonally additive polynomials in Riesz spaces in analogy to the representation in Banach lattices.

The representation theorem  \cite{Benyamini-Lassalle-Llavona 2006}  for orthogonally $n$-homogeneous additive polynomials on a Banach lattice $E$ with values on a Banach space $F$, $\mathcal{P}_0({}^nE, F)$, establishes an isometry with the space of $F$--valued continuous  linear maps on the $n$-concavification of $E$, $\mathcal{L}(E_{(n)},F)$, given by $P(f) = T(f^n)$ for all $f\in E$, $P\in\mathcal{P}_0({}^nE,F)$ and $T \in \mathcal{L}(E_{(n)},F)$.
The aim of this article is to provide a representation theorem for orthogonally additive polynomials on Riesz spaces similar to this representation theorem.

Recall that a real vector space $E$ together with an order relation $\leq$  compatible with the algebraic operations in $E$ is called a Riesz space if, for every $x,y\in E$, there exists a least upper bound $x\v y$ and a greatest lower bound $x\w y$. We define the absolute value of $x\in E$ as $|x|=x\v (-x)$. An element $x\in E$ is said to be
positive if $x\geq 0$. The positive cone is the space     $$E^+=\{x\in E: x\geq 0\}.$$

Note that every $x\in E$ can be decomposed as the difference of two positive elements $x^+=x\v 0$ and $x^-=(-x)\v 0$. Furthermore, $|x|=x^++x^-$. A Riesz space will be called Archimedean if for every $x,y\in E^+$, such that $0\leq nx\leq y$ for every $n\in\N$ we have that $x=0$. For further information about Riesz spaces the reader is referred to the classical books \cite{Aliprantis-Bourkinshaw 2006} or \cite{Jonge-Van Rooij 1970}.

A mapping between two Riesz spaces $P:E\lra F$ is said to be a $n$-homogeneous polynomial if there exists a $n$-linear form $A:E\x\dots\x E\lra F$ such that $P(x)=A(x,\dots,x)$. There is a natural one-one correspondence between polynomials and symmetric $n$-linear forms.

Given $E_1,\dots,E_n,F,$ Riesz spaces, a multilinear form $A:E_1\x\dots\x E_n\lra F$ is said to be positive if $A(x_1,\dots,x_n)\geq 0$ for every $x_1,\dots, x_n$ positive elements. A polynomial is positive if its associated symmetric multilinear mapping is positive.

\section{Orthosymmetric Applications on Riesz spaces.}\label{Ortosymmetric}

Our proof of the representation theorem relies on the connection between the orthogonally additive polynomials and the orthosymmetric multilinear mappings. Two ingredients will be essential: we will need the equivalence of these concepts for a polynomial and its associated multilinear form and we will also need that the orthosymmetry and the positiveness guarantee the symmetry of a multilinear form.

Following Boulabiar and Buskes \cite{Boulabiar-Buskes 2006} a multilinear map $A:E\x\dots\x E\lra F$ between two Riesz spaces $E, F$ is said to be orthosymmetric if  $A(x_1,\dots, x_n)=0$ whenever $x_1,\dots, x_n\in E$ verifies $|x_i|\w |x_j|=0$ for some $i,j \in \{1,\dots,n\}$. 

The next result shows the relation between the orthosymmetry and the additive orthogonality using the proof of
the already quoted representation theorem \cite{Benyamini-Lassalle-Llavona 2006}. We will need a technical lemma.

\begin{Lemma}
Let $K$ be a compact Hausdorff space, $F$ be a Banach space and $P:C(K)\lra F$ a continuous polynomial. If $P(f+g)=P(f)+P(g)$ whenever $f,g\in C(K)$ verify $supp(f)\cap supp(g)=\emptyset$, then $P$ is orthogonally additive.
\end{Lemma}

\begin{proof}
Let $f, g\in C(K)$ be two disjoint elements. Define for every $n \in \N$ the compact set $$K_n=\{x\in K: |f(x)|\geq 1/n\}$$ and a continuous function $\Phi_n: K\lra \R$
such that $0\leq \Phi_n\leq 1$, $\Phi_n=1$ on $K_n$ and $\Phi_n=0$ on $K-K_{n+1}$. Then $f\Phi_n$ converges uniformly to $f$ and $P(f)=\lim_n P(f\Phi_n)$.
Perfom the same construction for $g$ to get a sequence of functions $\Psi_n: K\lra \R$ with analogous properties such that $g\Psi_n$ converges uniformly to $g$ and $P(g)=\lim_n P(g\Psi_n)$. 

As $f$ and $g$ are disjoint, for every $n\in \N $, $supp(f\Phi_n)\cap supp(g\Psi_n)=\emptyset$ and hence $$P(f\Phi_n+g\Psi_n)=P(f\Phi_n)+P(g\Psi_n).$$

Taking limits, $P(f+g)=P(f)+P(g)$, so $P$ is orthogonally additive.
\end{proof}

\begin{Th}\label{Equiv definicion BB}
Let $E$ be a Riesz space, $F$ a Banach space and $A:E\times\dots\times E\lra F$ a
symmetric positive multilinear form in $E$. $A$ is
orthosymmetric if and only if $P=\hat{A}$ is orthogonally
additive.
\end{Th}

\begin{proof}
It is clear that if $A$ is orthosymmetric, expanding
$$P(x+y)=A(x+y,\dots,x+y)$$ for disjoint $x$ and $y$, the unique
nonzero terms are $A(x,\dots,x)$ and $A(y,\dots,y)$ so $P$ is
orthogonally additive.

Assume that $P$ is orthogonally additive. Let $x_1,\dots,x_n$ be
elements of $E$ at least two of which are disjoint. Consider $E_0\subset
E$ to be the ideal generated by $x_1,\dots,x_n$ whose unit is
$|x_1|+\dots +|x_n|$. By Yosida's Representation Theorem (see for instance
\cite{Jonge-Van Rooij 1970} Theorem 13.11), $E_0$ is Riesz isomorphic to a dense
ideal $\hat{E_0}$ of $C(K)$ for certain compact Hausdorff $K$.
Let $\hat{P}$ be the polynomial on $\hat{E_0}$ defined by the
composition of the previous isomorphism with $P$. As the
isomorphism preserves the order, $\hat{P}$ is positive and
orthogonally additive. 

Extend $\hat{P}$ by denseness to a positive orthogonally additive $n$-homogeneous polynomial  $\tilde{P}$ on $C(K)$.  For this purpose, by the previous lemma, we just have to show for $f,g\in C(K)$ with $supp(f)\cap supp(g)=\emptyset$ that $\tilde{P}(f+g)=\tilde{P}(f)+\tilde{P}(g)$. Take such $f$ and $g$ and denote by $K_1=supp(f)$ and by $K_2=supp(g)$. As these sets are compact, there exists two open sets $U_1, U_2$ on $K$ such that $K_i\subset U_i$, $i=1,2$ and $\overline{U}_1\cap \overline{U}_2=\emptyset$. Define two functions $\Phi_i\in C(K)$ such that $0\leq \Phi_i\leq 1$, $\Phi_i=1$ on $K_i$ and $\Phi_i=0$ on $K- U_i$, $i=1,2$.

Let be $x_n, y_n \in  \hat{E_0}$ converging to $f$ and $g$ respectively. As $\hat{E_0}$ is an ideal, for every $n\in \N$, $x_n\Phi_1, y_n\Phi_2\in \hat{E_0}$. As $x_n\Phi_1$ and $y_n\Phi_2$ converge respectively to $f$ and $g$ and verify $supp(x_n\Phi_1)\cap supp(y_n\Phi_2)=\emptyset$, by the orthogonality of $\hat{P}$ on  $\hat{E_0}$, $$\hat{P}(x_n\Phi_1+y_n\Phi_2)=\hat{P}(x_n\Phi_1)+\hat{P}(y_n\Phi_2)$$
taking limits $\tilde{P}(f+g)=\tilde{P}(f)+\tilde{P}(g)$ and then $\tilde{P}$ is orthogonally additive. Note that this extension, being positive, will also be continuous (see
\cite{Grecu-Ryan 2004} Proposition 4.1).

Hence, by the representation theorem of orthogonally additive polynomials \cite{Benyamini-Lassalle-Llavona 2006} and taking into account that $C(K)_{(n)} = 
C(K)$ there exists $T \in \mathcal{L}(C(K)_{(n)}, F)$  
such that
    $$\tilde{P}(f)= T(f^n) .$$
The symmetric $n$-linear form associated to $\tilde{P}$ is
given by
    $$(f_1,\dots,f_n)\longmapsto T(f_1\cdots f_n),$$
then, this form as well as its restriction to $\hat{E_0}$ are
orthosymmetric and by the Riesz isomorphism between $E_0$ and
$\hat{E_0}$, $A$ will also be orthosymmetric.
\end{proof}

In spite of this result we are still lacking of a natural and explicit representation of an orthogonally
additive polynomial on a Riesz space as a linear form on a Riesz space.     Such space would 
play the role of the $n$--concavification of Banach lattices used in the representation
theorem for orthogonally additive polynomials on Banach lattices.  The notion of $n$-power of
Riesz spaces introduced by Boulabiar and Buskes \cite{Boulabiar-Buskes 2006} will be the appropriate tool for that.

We introduce a new
definition of orthosymmetry with the properties needed to
actually prove our claims.   The definition of orthosymmetry that we are going
to use is more involved than the definition of orthosymmetry above but it
will lead us to the proof the the main statement (see Theorem \ref{n-ortos=ortog aditivo} below). 

To begin with, we introduce the concept of $p$-disjoint
elements:

\begin{Def}
We will say that $x_1,\dots,x_n\in E$ are partitionally disjoint
or p-disjoint if there exists a partition $I_1,\dots,I_m$, $2\leq
m\leq n$, of the index set $I=\{1,\dots,n\}$ such that the sets
    $$\{x_{i_k}: i_k\in I_k\}$$
are disjoint, that is $|x_{i_k}|\wedge|x_{i_l}|=0$ whenever $k\neq
l $.
\end{Def}

\begin{Def}\index{Orthosymmetric mapping}
Given $E,F$ Riesz spaces, a $n$-linear application
$A:E\times\dots\times E\lra F$ is said to be $p$-orthosymmetric if
$A(x_1,\dots,x_n)=0$ for $x_1,\dots,x_n$ $p$-disjoint.
\end{Def}

In order to prove that this new definition is coherent with the
additive orthogonality of polynomials, we need a preliminary lemma:

\begin{Lemma}
Let $E, F$ be Riesz spaces and $A:E\times\dots\times E\lra F$ a
symmetric $n$-linear form. The following assumptions are
equivalent
\begin{enumerate}
\item $A$ is $p$-orthosymmetric.
\item $A(x^i,y^{n-i})=0$ for $x$ and $y$ disjoint and $1<i<n$.
\end{enumerate}
\end{Lemma}
\begin{proof}
If $A$ is $p$-orthosymmetric, it is obvious that
$A(x^i,y^{n-i})=0$ for disjoint $x$ and $y$ with $1<i<n$.

Conversely, let $\{x_1,\dots, x_n\}$ be $p$-disjoint. Assume
that the partition in disjoint subsets is given by two elements,
that is there exists some $1<i<n$ such that the sets
$\{x_1,\dots,x_i\}$ and $\{x_{i+1}\dots,x_n\}$ are disjoints. The
general case is analogous.

The notation $B=A_{x_{i+1},\dots, x_n}$ will represent the
multilinear form $B$ defined by $B(y_1,\dots, y_i)=A(y_1,\dots,
y_i,x_{i+1},\dots, x_n)$. By the Polarization Formula,

$$
A(x_1,\dots,x_n)=B(x_1,\dots,x_i)=\frac{1}{i!2^i}\sum_{\e_j=\pm
1}\e_1\cdots \e_i B(\e_1 x_1+\dots +\e_ix_i)^i
$$
where $B(x)^i=B(x,\overbrace{\dots}^i,x)$.

We conclude by showing that for each choice of signs $\e_j=\pm 1$,
$$B(\e_1 x_1+\dots+ \e_i x_i)^i=0.$$

In order to simplify, let $ \e \bx=\e_1 x_1+\dots +\e_i x_i$, note
that

  \begin{eqnarray*}
    B(\e\bx)^i&=&B(\e\bx,\dots,\e\bx)=
    A(\e\bx,\dots,\e\bx,x_{i+1},\dots,x_n)=\\
    & & A_{\e\bx,\dots,\e\bx}(x_{i+1},\dots,x_n)=C(x_{i+1},\dots,x_n)
\end{eqnarray*}
with the notations above.

Again by using the polarization formula, we get:

 $$C(x_{i+1},\dots,x_n)= \frac{1}{(n-i)!2^{n-i}}\sum_{\delta_k=\pm 1}
\delta_1\cdots\delta_{n-i}C(\delta_1
x_{i+1}+\dots+\delta_{n-i}x_n)^{n-i}$$ as $\e_1 x_1+\dots +\e_i
x_i$ and $\delta_1 x_{i+1}+\dots + \delta_{n-i} x_n$ are disjoint
for every choice of signs $\e_j=\pm 1$ and $\delta_k=\pm 1$ our
assumption allows us to conclude that the summands appearing in
the previous expressions are zero and hence $A(x_1,\dots,x_n)=0$,
as needed.
\end{proof}

The next result can be obtained as a straightforward generalization of Proposition 2.4 on \cite{Perez Garcia-Villanueva 2005}.

\begin{Th}\label{n-ortos=ortog aditivo}
Let $E,F$ be Riesz spaces and $A:E\times\dots\times E\lra F$ a
symmetric multilinear form in $E$. $A$ is $p$-orthosymmetric if
and only if $P=\hat{A}$ is orthogonally additive.
\end{Th}

Finally, we will see that the $p$-orthosymmetry and the
positiveness guarantee the symmetry as in the bilinear case. The
following lemma generalizes Theorem 1 in \cite{Buskes-Van Rooij1
2000}. To prove it we will follow a similar path from which it
will be clear the reasons for our definition of orthosymmetry.

\begin{Lemma}\label{Lemma Symmetry}
Let $K$ be a compact Hausdorff space, $E$ a uniformly dense Riesz
subspace of $C(K)$, $F$ an Archimedean Riesz space and
$T:E\x\dots\x E\lra F$ a positive $p$-orthosymmetric $n$-linear
map. Let $E^n$ the linear hull of the set
    $$\{f_1\cdots f_n: f_1,\dots, f_n\in E\}.$$
Then there exists a positive linear map $A: E^n\lra F$ such that

    $$T(f_1,\dots,f_n)=A(f_1\cdots f_n) \textrm{ for every } f_1,\dots, f_n \in E.$$
\end{Lemma}

\begin{proof}
As in \cite{Buskes-Van Rooij1 2000} there is no loss of generality
if we consider $E=C(K)$. Define
$A(h)=T(\mathbf{1},\dots,\mathbf{1},h)$ where $\mathbf{1}$ stands
for the function identically 1. It is clear that $A$ is linear and
positive. Note that as $T$ is positive by \cite{Grecu-Ryan 2004}
Proposition 4.1, it is continuous. We have to prove that if
$f_1,\dots, f_n\in C(K)$ and $h=f_1\cdots f_n$,
$T(f_1,\dots,f_n)=T(\mathbf{1},\dots,\mathbf{1},h)$.

Given $\e>0$, we will say, following \cite{Buskes-Van Rooij1
2000}, that $X\subset K$ is small if for every $x,y\in X$
    $$|f_1(x)-f_1(y)|<\e, \dots, |f_n(x)-f_n(y)|<\e \textrm{ and } |h(x)-h(y)|<\e.$$

Take $u_1,\dots, u_N\in C(K)^+$ verifying $\sum u_n= \mathbf{1}$
and such that for every $j$, the set $S_j=\{x\in K: u_j(x)\neq
0\}$ is small and nonempty, and take $s_j\in S_j$, for $j=1,\dots,
N$. Define for $i=1,\dots, n$,

    $$f_i'=\sum_{j=1}^N f(s_j)u_j \textrm{ and } h'=\sum_{j=1}^N h(s_j)u_j.$$

Because
\begin{eqnarray*}
|T(f_1,\dots,f_n)-T(\mathbf{1},\dots,\mathbf{1},h))| &\leq &
 |T(f_1,\dots,f_n)-T(f_1',\dots,f_n')| \\
&+&|T(f_1',\dots,f_n')-T(\mathbf{1},\dots,\mathbf{1},h'))|  \\
&+&|T(\mathbf{1},\dots,\mathbf{1},h'))-T(\mathbf{1},\dots,\mathbf{1},h))|
\end{eqnarray*}
we have to bound these three terms. The second is immediate:

\begin{eqnarray*}
|T(f_1',\dots,f_n')-T(\mathbf{1},\dots,\mathbf{1},h'))| \leq
\|h-h'\|_\infty |T(\mathbf{1},\dots,\mathbf{1})|\leq \e
T(\mathbf{1},\dots,\mathbf{1}).
\end{eqnarray*}

For the first

\begin{eqnarray*}
&&|T(f_1,\dots,f_n)-T(f_1',\dots,f_n')|\leq
|T(f_1,\dots,f_n)-T(f_1',f_2\dots,f_n)|+ \\
&& |T(f_1',f_2,\dots,f_n)-T(f_1',f_2',f_3\dots,f_n)|+\dots +|T(f_1',\dots, f_{n-1}',f_n)\\
&& -T(f_1',\dots,f_n')|=\sum_{i=1}^n|T(f_1',\dots, f_{i-1}',f_i-f_i',\dots, f_n)|\leq\\
&&  \sum_{i=1}^n
\|f_1'\|_\infty\cdots\|f_i-f_i'\|_\infty\cdots\|f_n\|_\infty
T(\mathbf{1},\dots,\mathbf{1}) \leq\\
&& \e T(\mathbf{1},\dots,\mathbf{1}) \sum_{i=1}^n
\|f_1\|_\infty\cdot\overbrace{\dots}^{[i]}\cdot\|f_n\|_\infty
\end{eqnarray*}

With the notation $\overbrace{\dots}^{[i]}$ we mean that in the
product, the factor $i$ does not appear. Finally

{\setlength\arraycolsep{0pt}
\begin{eqnarray*}
&&|T(\mathbf{1},\dots,\mathbf{1},h'))-T(\mathbf{1},\dots,\mathbf{1},h))|
\leq \\
&& \sum_{j_1,\dots,j_n} |f_1(s_{j_1})\cdots
f_{n-1}(s_{j_{n-1}})-f_1(s_{j_n})\cdots
f_{n-1}(s_{j_n})||f_n(s_n)|T(u_{j_1},\dots, u_{j_n})\\
&&\leq \|f_n\|_\infty\sum_{j_1,\dots,j_n}|f_1(s_{j_1})\cdots
f_{n-1}(s_{j_{n-1}})-f_1(s_{j_1})\cdots
f_{n-2}(s_{j_{n-2}})f_{n-1}(s_{j_n})|T_{j_1,\dots,j_n} \\
&&+\|f_n\|_\infty\sum_{j_1,\dots,j_n}|f_1(s_{j_1})\cdots
f_{n-2}(s_{j_{n-2}})-f_1(s_{j_n})\cdots
f_{n-2}(s_{j_n})||f_{n-1}(s_{j_n})|T_{j_1,\dots,j_n}\\
&&\leq \|f_1\|_\infty\overbrace{\dots}^{[n-1]}\|f_n\|_\infty
\sum_{j_1,\dots,j_n}|f_{n-1}(s_{j_{n-1}})-f_{n-1}(s_{j_{n}})|
T_{j_1,\dots,j_n}\\
&& +\|f_{n-1}\|_\infty\|f_{n-1}\|_\infty
\sum_{j_1,\dots,j_n}|f_1(s_{j_1})\cdots
f_{n-2}(s_{j_{n-2}})-f_1(s_{j_n})\cdots
f_{n-2}(s_{i_n})|T_{j_1,\dots,j_n}
\end{eqnarray*}
}

Where the notation $T_{j_1,\dots,j_n}$ stands for
$T(u_{j_1},\dots, u_{j_n})$. If we repeat the process, we get:
{\setlength\arraycolsep{0pt}
\begin{eqnarray*}
&&|T(\mathbf{1},\dots,\mathbf{1},h'))-T(\mathbf{1},\dots,\mathbf{1},h))|
\leq \\
&& \sum_{i=1}^n
\|f_1\|_\infty\overbrace{\dots}^{[i]}\|f_n\|_\infty
\sum_{j_1,\dots,j_n}|f_{i}(s_{j_{i}})-f_{n-1}(s_{j_{n}})|
T(u_{j_1},\dots, u_{j_n})
\end{eqnarray*}
}

Now, there are two options, if $u_{j_1},\dots, u_{j_n}$ are
$p$-disjoint, then $$T(u_{j_1},\dots, u_{j_n})=0,$$ otherwise
there is a path connecting the set $S_{j_i}$ and the set
$S_{j_n}$, hence
    $$|f_{i}(s_{j_{i}})-f_{n-1}(s_{j_{n}})|\leq n\e$$

We conclude that

{\setlength\arraycolsep{0pt}
\begin{eqnarray*}
&&|T(\mathbf{1},\dots,\mathbf{1},h'))-T(\mathbf{1},\dots,\mathbf{1},h))|
\leq \\
&& \sum_{i=1}^n
n\e\|f_1\|_\infty\overbrace{\dots}^{[i]}\|f_n\|_\infty
\sum_{j_1,\dots,j_n} T(u_{j_1},\dots, u_{j_n})\\
&&\e T(\mathbf{1},\dots,\mathbf{1})\sum_{i=1}^n
n\|f_1\|_\infty\overbrace{\dots}^{[i]}\|f_n\|_\infty
\end{eqnarray*}
}
\end{proof}

As a direct consequence of the previous lemma we obtain the
following result:

\begin{Th}\label{n-ortos y posi= simetrico}
Let $E$ and $F$ be Archimedean Riesz spaces. Let $T:E\x\dots\x
E\lra F$ be an $p$-orthosymmetric positive $n$-linear map. Then
$T$ is symmetric.
\end{Th}

The proof of this fact is an easy application of Yosida's Theorem which allows us to translate the problem to
the $C(K)$ setting and using then Lemma \ref{Lemma Symmetry}. 

\section[Representation of polynomials]{Representation of orthogonally additive polynomials.}

We will discuss now the notion of $n$--power of Riesz spaces as presented by Boulabiar and Buskes that
will be instrumental in what follows.

\begin{Def}[\cite{Boulabiar-Buskes 2006}]
Let $E$ be an Archimedean Riesz space and $n\geq 2$. The pair
$(\bod_n E,\od_n)$ it is said to be a $n$-power of $E$ if
\begin{enumerate}
\item $\bod_n E$ is a Riesz space,
\item $\od_n:E\x\dots\x E\lra \bod_n E$ is an orthosymmetric $n$-morphism of Riesz spaces
and,
\item for every $F$ Archimedean Riesz space and every orthosymmetric
$n$-morphism $T:E\x\dots\x E\lra F$, there exists a unique Riesz
homomorphism $ T^{\od_n}:\bod_n E\lra F$ such that $T=T^{\od_n}
\circ \od_n$
\end{enumerate}
\end{Def}

As those authors proved, the $n$-power is unique up to Riesz isomorphism. The results about the $n$-power can be translated in a straighforward way to the setting of $p$-orthosymmetry.

\begin{Lemma}\label{Lemma representacion}
Let $E$ be an Archimedean Riesz space and $F$ a uniformly
complete Archimedean Riesz space. The space $\Li^+_o(^nE,F)$ of
positive $p$-or\-tho\-sy\-mme\-tric $n$-linear applications from
$E$ to $F$ is isomorphic to the space $\Li^+(\bod_nE,F)$ of
positive linear forms from $\bod_nE$ to $F$.
\end{Lemma}

\begin{proof}
In this result, we generalize the ideas of \cite{Buskes-Kusraev
2007} Theorem 3.1 to the $n$-linear case. If $T\in \Li^+_o(^nE,F)$
we will prove that there exists a unique $\Phi_T\in
\Li^+(\bod_nE,F)$ such that $\Phi_T\circ\bod_n=T$ and that the
correspondence $T\lra \Phi_T$ is a Riesz isomorphism.

Given such $T$, let $\widetilde{T}$ be the unique positive
operator $\widetilde{T}: \overline{\bigox}_n E\lra F$ verifying
    $$T(x_1,\dots,x_n)=\widetilde{T}(x_1\ox\dots\ox x_n).$$

If $\pi: \overline{\bigox}_n E\lra \bod_nE$ is the canonical
projection, the operator $\Phi=\Phi_T$ that we need is the one
satisfying $\widetilde{T}=\Phi\circ\pi$. The uniqueness of this
operator is given by the uniqueness of $\widetilde{T}$. Moreover
$\widetilde{T}$ it is positive since $\pi$ is a Riesz homomorphism
and $\widetilde{T}$ is positive. In particular, the correspondence
$T\lmt \Phi$ respects the order.

To conclude the proof, we need to show that $\Phi$ is well
defined. It is sufficient to show that the kernel of $\pi$ is
contained in the kernel of $\widetilde{T}$. We proceed as in
\cite{Buskes-Van Rooij 2004} Theorem 4. Let $f_1\ox\dots\ox f_n
\in \ker \pi$ then $f_1,\dots,f_n$ are $p$-disjoint hence,
    $$\widetilde{T}(f_1\ox\dots\ox f_n)=T(f_1,\dots,f_n)=0$$
since $T$ is $p$-orthosymmetric.

Finally, observe that the correspondence $T\lmt \Phi_T$ is a Riesz
isomorphism since $T$ and $\Phi$ are positive.

\end{proof}

Notice the relevance of the notion of $p$-orthosymmetry in the proof of the previous lemma.
We obtain as a  consequence of the previous results the
representation theorem we were looking for:

\begin{Th}[Representation Theorem for positive orthogonally additive polynomials on Riesz spaces]
Let $E,F$ be Archimedean Riesz spaces with $F$ uniformly complete
and let $(\bod_nE,\od_n)$ be the $n$-power of $E$. The space
$\Po_o^+(^nE,F)$ of positive orthogonally additive $n$-homogeneous
polynomials on $E$ is isomorphic to the space $\Li^+(\bod_n
E,F)$ of positive linear applications from $\bod_n E$ to $F$.
\end{Th}

If we consider the case of $f$-algebras, the connection with the
theorem of representation of orthogonally additive polynomials in
Banach lattices becomes clearer:

\begin{Th}
Let $E$ be a uniformly complete Archimedean Riesz subspace of a
semiprime $f$-algebra $A$ and $F$ uniformly complete Archimedean.
Then, for every positive orthogonally additive $n$-homogeneous
polynomial $P \in \Po_o^+(^nE,F)$ there exists a unique positive
linear application $L\in \Li^+(E^n,F)$ such that $P(x)=L(x\cdots
x)$ for every $x\in E$.
\end{Th}

\begin{proof}
By Theorem 3.3 \cite{Boulabiar-Buskes 2006}, $E^n$ with its product is
the $n$-power of $E$ and by the uniqueness given by Theorem 3.2 \cite{Boulabiar-Buskes 2006}, there exists a Riesz isomporphism
$i:E^n\lra \bod_n E$ such that $i(x_1\cdots x_n)=x_1\od\dots\od
x_n$. Define $L\in \Li^+(E^n,F)$ as the linear application
$L=\Phi_{\check{P}}\circ i$ (with the notation as in Lemma
\ref{Lemma representacion}) then,
$$P(x)=\check{P}(x,\dots,x)=\Phi_{\check{P}}(x\od\dots\od
x)=L(x\cdots x).$$
\end{proof}

Recently Toumi \cite{Toumi} has obtained an alternative representation theorem for homogeneous orthogonally additive polynomials on Riesz spaces.   Its relation with the results presented here will be discussed elsewhere.

%%%%%%%%%%%%%%%%%%%%%%%%%%%BIBLIOGRAFIA%%%%%%%%%%%%%%%%%%%%%%%%%%%%%

\end{document}